\documentclass[11pt]{article}
\usepackage{amsmath}
\usepackage{amsthm}
\usepackage{epsf}
\usepackage{amsfonts}
\usepackage{graphicx}
\usepackage{multirow}
\usepackage{subcaption}
\usepackage{enumitem}
\usepackage{todonotes}
\usepackage{comment}
\usepackage{float}
\usepackage{xcolor}

\newcommand{\sK}{\mathcal{K}}

\newcommand{\sM}{\mathcal{M}}

\DeclareMathOperator{\nullity}{nullity}
\DeclareMathOperator{\range}{range}
\DeclareMathOperator{\rank}{rank}

\DeclareMathOperator{\diag}{diag}
\DeclareMathOperator{\IC}{IC}
\DeclareMathOperator{\nnz}{nnz}

\newtheorem{theorem}{Theorem}
\newtheorem{lemma}[theorem]{Lemma}

\newtheorem{proposition}[theorem]{Proposition}

\newcommand*{\overtabline}{%
	\noalign{%
		\vskip-.5\dimexpr\ht\@arstrutbox+\dp\@arstrutbox\relax
		\vskip-.2pt\relax
		\hrule
		\vskip-.2pt\relax
		\vskip+.5\dimexpr\ht\@arstrutbox+\dp\@arstrutbox\relax
	}%
}

\begin{document}
	\author{Susanne Bradley\thanks{Department of Computer Science, The University of British Columbia, Vancouver, Canada V6T 1Z4
			({smbrad@cs.ubc.ca}, {greif@cs.ubc.ca}).}
		\and Chen Greif\footnotemark[2]}
	\title{Augmentation-Based Preconditioners for Saddle-Point Systems with Singular Leading Blocks}
		
	\maketitle
		
	\begin{abstract}
		We consider the iterative solution of symmetric saddle-point matrices with a singular leading block. We develop a new ideal positive definite block diagonal preconditioner that yields a preconditioned operator with four distinct eigenvalues. We offer a few techniques for making the preconditioner practical, and illustrate the effectiveness of our approach with numerical experiments.
	\end{abstract}
		
	%\begin{keywords}
	%		saddle-point systems, eigenvalue bounds, eigenvalues and singular values
	%\end{keywords}

	\section{Introduction}
	\label{sec:intro}
	Consider the saddle-point system
	\begin{equation}
		\label{eq:sp-system}
		\begin{bmatrix}
			A & B^T \\
			B & 0
		\end{bmatrix}
		\begin{bmatrix}
			x \\
			y
		\end{bmatrix} =
		\begin{bmatrix}
			f \\
			g
		\end{bmatrix},
	\end{equation}
	where $A \in \mathbb{R}^{n \times n}$ is symmetric positive semidefinite and $B \in \mathbb{R}^{m \times n}$ has full row rank, with $m < n$. We denote the coefficient matrix by
	\begin{equation*}
		\sK = \begin{bmatrix}
			A & B^T \\
			B & 0
		\end{bmatrix}.
	\end{equation*}
	We assume throughout that $\sK$ is invertible. 
	A necessary and sufficient condition for this is that $\ker(A) \cap \ker(B) = \{ 0\}$; see \cite[Theorem 3.2]{bgl05}. Thus, the nullity of $A$ must be no greater than $m$, or $\sK$ will necessarily be singular. We therefore say that a leading block $A$ with nullity $m$ is \textit{lowest-rank} or \textit{maximally rank-deficient}.
	Under the assumptions above, the matrix $\sK$ is symmetric and indefinite, and the solution of the linear system \eqref{eq:sp-system} poses several numerical challenges; we refer to the survey of \cite{bgl05} for an overview of solution methods. 

Our focus is on positive definite preconditioners, which maintain symmetry of the preconditioned operator and can therefore be used with a symmetric iterative solver such as MINRES \cite{p75}.	
	When $A$ is positive definite, the preconditioner of Murphy, Golub, and Wathen \cite{mgw00}
	\begin{equation*}
	\sM_1 = \begin{bmatrix}
		A & 0 \\
		0 & BA^{-1}B^T
	\end{bmatrix}
	\end{equation*}
	has the property that the preconditioned operator $\sM_1^{-1}\sK$ has three distinct eigenvalues, meaning that a preconditioned iterative solver (such as MINRES) will converge within three iterations, in exact arithmetic. In practice, the matrices $A$ and $BA^{-1}B^T$ are too expensive to form and solve for exactly, so approximations must be sought.
	
	The case in which $A$ is singular has been less studied; see \cite{eg15,ggv05,gs07}
	for preconditioning approaches in this setting. Golub, Greif, and Varah \cite{ggv05} have analyzed the positive definite block diagonal preconditioner
	\begin{equation*}
		\sM_2 = \begin{bmatrix}
			A + B^T W B & 0 \\
			0 & B(A+B^TWB)^{-1}B^T
		\end{bmatrix},
	\end{equation*}
	where $W \in \mathbb{R}^{m \times m}$ is a positive semidefinite matrix such that $A + B^T WB$ is positive definite. This can be considered a generalization of $\sM_1$, in which a semidefinite term is first added to the leading block to make it positive definite. Because of the requirement that $\ker(A) \cap \ker(B) = \{0\}$, the matrix $A+B^T W B$ is necessarily positive definite if $W$ is positive definite (though this is not a necessary condition unless $A$ is lowest-rank).
	
	While the preconditioned operator $\sM_2^{-1}\sK$ is not guaranteed to have a fixed, small number of distinct eigenvalues, it is shown in \cite[Theorem 2.5]{ggv05} that the eigenvalues are bounded within the intervals $\left[ -1, \frac{1-\sqrt{5}}{2} \right] \cup \left[ 1, \frac{1+\sqrt{5}}{2} \right]$. However, from \cite[Theorem 3.5]{eg15} and \cite[Theorem 4.1]{gs07}, we can observe that $\sM_2^{-1}\sK$ does have exactly two distinct eigenvalues when $A$ has maximal nullity.
	
	{\bf Contribution of this paper.}
	At present, the literature provides ideal positive definite block diagonal preconditioners that yield preconditioned operators with a small number of distinct eigenvalues (and, therefore, will lead to convergence of a preconditioned iterative solver in a small number of iterations in the absence of round-off error) in the cases where $A$ has full rank and where $A$ has maximal nullity. In this work, we bridge the gap between the full-rank and minimal-rank (or maximal-nullity) cases by providing such a preconditioner for cases in which $(n-m) < \rank(A) < n$. This is meaningful because on the one hand we cannot invert $A$ and given its assumed rank deficiency, the Schur complement $B A^{-1} B^T$ does not exist either, making it difficult to develop standard preconditioners. And on the other hand unique algebraic properties that have been studied in \cite{eg15,ggv05,gs07} for the maximal-nullity case cannot be applied either.
	
	{\bf Outline.}  We provide relevant mathematical background in Section \ref{sec:background} and describe our preconditioning approach in \ref{sec:prec_ours}. We then provide numerical experiments in Section \ref{sec:numex} and concluding remarks in Section \ref{sec:conclusions}.
	
	\section{Mathematical background}
	\label{sec:background}
	In this section, we provide some existing results that will aid us in developing and analyzing our preconditioner. Section \ref{sec:aug} describes previous strategies in the literature for augmenting a rank-deficient leading block $A$, and Section \ref{sec:mrd_properties} describes some special properties of matrices with maximally rank-deficient leading blocks. We then use these techniques to provide an alternative proof of a result in \cite{gs07} for matrices with a maximally rank-deficient $A$, and we use the insights of this alternative proof to adapt this approach to matrices with non-maximally rank-deficient $A$ in Section \ref{sec:prec_ours}.
	
	\subsection{Leading block augmentation}
	\label{sec:aug}
	Our strategy for preconditioning involves augmenting the leading block $A$ so that it becomes positive definite, rather than semidefinite. We observe that \eqref{eq:sp-system} can be reformulated as (see, for example, \cite{f75,gg03}):
	\begin{equation*}
		\begin{bmatrix}
			A+B^TWB & B^T \\
			B & 0
		\end{bmatrix}\begin{bmatrix}
		x \\
		y
	\end{bmatrix} = \begin{bmatrix}
	f + B^TWg \\
	g
	\end{bmatrix},
	\end{equation*}
	where $W$ is an $m \times m$ matrix. We will assume $W$ is positive semidefinite and the leading block 
	\begin{equation}
	A_W = A+B^TWB
	\label{eq:ABTWB}
	\end{equation} 
	is positive definite. An advantage of this approach is that a positive definite leading block will provide flexibility in both forming and analyzing our preconditioners later in this paper. This approach was effective used in \cite{bo06} for fluid flow problems. We also recall the following result \cite{f75,gg03}:
	\begin{lemma}
		\label{lem:kw}
		Let
		\begin{equation*}
			\sK(W) = \begin{bmatrix}
				A_W & B^T \\
				B & 0
			\end{bmatrix},
		\end{equation*}
		where $W \in \mathbb{R}^{m \times m}$. If $\sK$ and $\sK(W)$ are both nonsingular, then
		\begin{equation*}
			\sK^{-1} = (\sK(W))^{-1} + \begin{bmatrix}
				0 & 0 \\
				0 & W
			\end{bmatrix}.
		\end{equation*}
	\end{lemma}
	
	\subsection{Matrix properties when $\nullity(A) = m$}
	\label{sec:mrd_properties}
	When $A$ has maximal nullity -- that is, when $\nullity(A) = m$ -- the blocks of $\sK$ and those of the augmented matrix $\sK(W)$ interact in unique ways, which provide useful tools in the design and analysis of preconditioners.
	
	Estrin and Greif \cite[Theorem 3.5]{eg15} provide the following result on the Schur complement of $\sK(W)$:
	\begin{proposition}
		\label{lem:mrd_schur}
		Suppose $\nullity(A) = m$ and let $W \in \mathbb{R}^{m \times m}$ be an invertible matrix. Then
		$$
		B(A+B^TWB)^{-1}B^T = W^{-1}.
		$$
	\end{proposition}
	
	We also recall the following result \cite[Corollary 2.1]{eg16} applying to more general matrices, which we will use repeatedly in our analyses:
	\begin{lemma}
		\label{lem:MandN}
		For matrices $M, N \in \mathbb{R}^{n \times n}$ with $\rank(M) = r, \rank(n) = n-r$ and $M+N$ nonsingular, the matrix $(M+N)^{-1}M$ is a projector with rank $r$. Moreover,
		$$
		M(M+N)^{-1}N = 0.
		$$
	\end{lemma}

	A recent preprint \cite{bg22_eig} provides eigenvalue bounds for saddle-point systems with a rank-deficient leading block. We will use the following result \cite[Theorem 7]{bg22_eig} in our analyses:
	\begin{theorem}
		\label{thm:mrd_K_bnd}
		When $\rank(A) = n-m$, the positive eigenvalues of $\sK$ are greater than or equal to
		$$
		\min\left\{ \mu_{\min}^+ (1- \cos(\theta_{\min})), \sigma_{\min}\sqrt{1-\cos(\theta_{\min})} \right\}
		$$
		where: $\mu_{\min}^+$ denotes the smallest positive eigenvalue of $A$; $\sigma_{\min}$ the smallest singular value of $B$; and $\theta_{\min}$ the minimum principal angle between $\range(A)$ and $\range(B^T)$.
	\end{theorem}

	\subsection{Preconditioning when $\nullity(A) = m$}
	\label{sec:mrd_prec}
	We consider the block diagonal preconditioner \cite{gs07}
	\begin{equation}
		\label{eqn:sM}
		\sM_W = \begin{bmatrix}
			A_W & 0 \\
			0 & W^{-1}
		\end{bmatrix},
	\end{equation}
	where $W$ is positive definite and $A_W$ is as defined in~\eqref{eq:ABTWB}. Let us denote the blocks of the split preconditioned operator $\sM_W^{-1/2}\sK \sM_W^{-1/2}$ as follows:
	\begin{equation*}
		\sM_W^{-1/2}\sK \sM_W^{-1/2} = \begin{bmatrix}
			A_W^{-1/2}A A_W^{-1/2} & A_W^{-1/2}B^T W^{1/2} \\
			W^{1/2} B A_W^{1/2} & 0
		\end{bmatrix}
		=: \begin{bmatrix}
			\tilde{A} & \tilde{B}^T \\
			\tilde{B} & 0
		\end{bmatrix}.
	\end{equation*}
	
	\begin{lemma}
		\label{lem:wishlist}
		When $\rank(A) = n-m$, the blocks of $\sM_W^{-1/2}\sK \sM_W^{-1/2}$ satisfy the following:
		\begin{enumerate}[label={(\roman*)}]
			\item All nonzero eigenvalues of $\tilde{A}$ are equal to 1;
			\item All singular values of $\tilde{B}$ are equal to 1;
			\item The subspaces $\range(\tilde{A})$ and $\range(\tilde{B}^T)$ are orthogonal.
		\end{enumerate}
	\end{lemma}

	\begin{proof}
		To prove (i), we note that $\tilde{A}$ is similar to $A_W^{-1}A$, which is a projector by Lemma \ref{lem:MandN}. Lemma \ref{lem:mrd_schur} gives us that $BA_W^{-1}B^T = W^{-1}$, and therefore
		$$
		\tilde{B}\tilde{B}^T = W^{1/2}BA_W^{-1}B^TW^{1/2} = I,
		$$
		which proves (ii). We prove (iii) by showing that $\range(\tilde{B}^T) \in \ker(\tilde{A})$. We write
		$$
		\tilde{A}\tilde{B}^T = A_W^{-1/2}AA_W^{-1}B^T W^{-1/2} = 0,
		$$
		where the second inequality follows from the result of \cite[Proposition 2.6]{eg15}, which shows that $A_W^{-1}B^T$ is a null-space matrix of $A$.
	\end{proof}
	
	We now consider what the results of Lemma \ref{lem:wishlist} tell us about the eigenvalues of $\sM_W^{-1}\sK$ when $\rank(A) = n-m$. The orthogonality of $\range(\tilde{A})$ and $\range(\tilde{B}^T)$ means that the value of $\cos(\theta_{\min})$ in Theorem \ref{thm:mrd_K_bnd} is 1, and thus that the positive eigenvalues are greater than or equal to the minimum of the smallest positive eigenvalue of $\tilde{A}$ and the smallest singular value of $\tilde{B}$. These are both equal to 1, by parts (i)-(ii) of Lemma \ref{lem:wishlist}. Because the maximal eigenvalues of $\tilde{A}$ and singular values of $\tilde{B}$ are also equal to 1, all negative eigenvalues are equal to $-1$ and all positive eigenvalues are less than or equal to $1$ (as a consequence of Lemma 2.1 of Rusten and Winther \cite{rw92}). This yields the following result, which is also shown via a different proof method in \cite[Theorem 4.1]{gs07}; we refer to their proof for derivation of the multiplicities of the eigenvalues.
	\begin{proposition}
		\label{prop:gs07}
		When $\rank(A) = n-m$, the matrix $\sM_W^{-1}\sK$ has two distinct eigenvalues given by $1$ and $-1$ with algebraic multiplicities $n$ and $m$, respectively.
	\end{proposition}

	Proposition \ref{prop:gs07} tells us that, when $A$ has maximal nullity there is a block diagonal preconditioner that yields a preconditioned operator with two distinct eigenvalues. This is similar to the block diagonal preconditioner of \cite{mgw00}, which yields a preconditioner with three distinct eigenvalues in the case that $A$ is positive definite. What has not yet been developed is a preconditioner that gives a constant number of distinct eigenvalues for the ``in-between'' case where $A$ is rank-deficient, but not lowest-rank. This is the focus of the next section.

	\section{Block diagonal preconditioning for non-maximal nullity}
	\label{sec:prec_ours}
	
	\subsection{Preconditioner derivation}
	\label{sec:prec_derive}
	Let us now consider the case in which $A$ has nullity $k$, with $k < m$. We will now consider how we can devise a preconditioner to preserve (perhaps approximately) the properties listed in Lemma \ref{lem:wishlist} in the case where we no longer have maximal nullity.
	
	Let us consider a general block diagonal preconditioner of the form
	$$
	\sM = \begin{bmatrix}
		A+G & 0 \\
		0 & C
	\end{bmatrix},
	$$
	where $C$ is positive definite and $G$ is a semidefinite matrix such that $A+G$ is positive definite. As before, let us define the split preconditioned system:
	\begin{align*}
		%\label{eqn:split_prec_gen}
		\sM^{-1/2}\sK \sM^{-1/2} &= \begin{bmatrix}
			(A+G)^{-1/2}A (A+G)^{-1/2} & (A+G)^{-1/2}B^T C^{-1/2} \\
			C^{-1/2} B (A+G)^{-1/2} & 0
		\end{bmatrix}\\
		&=: \begin{bmatrix}
			\tilde{A} & \tilde{B}^T \\
			\tilde{B} & 0
		\end{bmatrix}.
	\end{align*}
	
	Property (i) of Lemma \ref{lem:wishlist} holds whenever $\rank(G) = k$; see Lemma \ref{lem:MandN}. It is also straightforward to verify, using a similar process as in the proof of Lemma \ref{lem:wishlist}, that Property (ii) holds if and only if
	$$
	C = B(A+G)^{-1}B^T.
	$$
	Property (iii) of Lemma \ref{lem:wishlist} holds because, in that Lemma's setting,
	$$
	A(A+G)^{-1}B^T = 0.
	$$
	We can write this as
	%\begin{subequations}
	\begin{align}
	\label{eq:bSub}
	\begin{aligned}
		A(A+G)^{-1}B^T &= (A+G-G)(A+G)^{-1}B^T  \\
		&= B - G(A+G)^{-1}B^T. 
	\end{aligned}
	\end{align}
	%\end{subequations}
	Suppose that $G$ has rank $k$, as we have already established will ensure Property (i). Then, as a consequence of Lemma \ref{lem:MandN}, $G(A+G)^{-1}$ is a projector onto the range of $G$. From \eqref{eq:bSub} we see that Property (iii) will hold if $G(A+G)^{-1}$ is a projector onto the range of $B^T$; however, this is clearly not possible if $\rank(G) = k < m$. But we note that if we set
	$$
	G = B^T W_k B,
	$$
	where $W_k$ is a symmetric positive semidefinite matrix of rank $k$, this matrix will be a projector onto a rank-$k$ subspace of $\range(B^T)$. While Property (iii) will not hold in this case because we will not have $\tilde{A}\tilde{B}^T = 0$, we instead have that $\nullity(\tilde{A}\tilde{B}^T) = k$ (which is the highest nullity we can achieve, as from \eqref{eq:bSub} we have a rank-$k$ term being subtracted from $B$).
	
	Thus, we consider the preconditioner:
	\begin{equation}
		\label{eqn:Mk}
		\sM_k = \begin{bmatrix}
			A_k & 0 \\
			0 & S_k
		\end{bmatrix},
	\end{equation}
	where $A_k = A + B^T W_k B$ and $S_k = BA_k^{-1}B^T$, with $\rank(W_k) = \nullity(A) = k$. This is the same preconditioner analyzed in \cite{ggv05}, but with the additional assumption that $\rank(W_k) = k$.
	
	\paragraph{Remark 1} We note that, when $A$ has maximal nullity, the preconditioner $\sM_k$ reduces to that of Greif and Sch\"{o}tzau defined in eq. \eqref{eqn:sM}. When $A$ is positive definite, then $\sM_k$ is equivalent to the preconditioner $\sM_1$.
	
	\subsection{Analysis of $\sM_k$}
	\label{sec:prec_analysis}
	We present some lemmas that will be necessary for our analysis. 
	
	\begin{lemma}
		\label{lem:sc_add}
		When $\rank(W_k) = \nullity(A) = k$,
		$$
		(BA_k^{-1}B^T)^{-1} = W_k + (BB^T)^{-1}B(A - AVA)B^T(BB^T)^{-1},
		$$
		where $V = Z(Z^T A Z)^{-1}Z^T$ with $Z \in \mathbb{R}^{n \times (n-m)}$ being a null-space matrix of $B$.
	\end{lemma}

	\begin{proof}
		The proof follows by considering the block inverses of $\sK$ and
		$$
		\sK(W_k) := \begin{bmatrix}
			A_k & B^T \\
			B & 0
		\end{bmatrix}.
		$$
		Let $Z \in \mathbb{R}^{n \times (n-m)}$ denote a matrix whose columns form a basis for $\ker(B)$. The inverse of $\sK$ is (see \cite[Eq. (3.8)]{bgl05}):
		$$
		\sK^{-1} = \begin{bmatrix} 
			V & (I-VA)B^T(BB^T)^{-1} \\
			(BB^T)^{-1}B(I-AV) & -(BB^T)^{-1}B(A-AVA)B^T(BB^T)^{-1}
		\end{bmatrix},
		$$
		where $V = Z(Z^TAZ)^{-1}Z^T$; we note that $Z^T A Z$ must be nonsingular for any nonsingular $\sK$ (see \cite{bgl05}). The result then follows from Lemma \ref{lem:kw} and the fact that the (2,2)-block of $(\sK(W_k))^{-1}$ is equal to $-(BA_k^{-1}B^T)^{-1}$ (see \cite[Eq. (3.4)]{bgl05}).
	\end{proof}

	\begin{lemma}
		\label{lem:commute}
		The matrix $VA$ is a projector. Moreover, when $\rank(W_k) = \nullity(A) = k$, the following results hold:
		\begin{enumerate}[label={(\roman*)}]
			\item The matrix $A_k^{-1}A$ is a projector;
			\item The matrices $VA$ and $A_k^{-1}A$ commute.
		\end{enumerate}
	\end{lemma}

	\begin{proof}
		By writing $VA = Z(Z^T A Z)^{-1}Z^TA$, it is clear that $VA$ is a projector onto $\ker(B)$. Item (i) holds because of Lemma \ref{lem:MandN}.
		
		To verify (ii), we first note that
		$$
		VAA_k^{-1} A = VA,
		$$
		because $AA_k^{-1}$ is a projector (this follows from the fact that $A_k^{-1}A = (AA_k^{-1})^T$ is a projector) onto the range of $A$. Because $A_k^{-1} A = I - A_k^{-1}B^T W_kB$, we can write
		$$
		A_k^{-1}AZ = Z - A_k^{-1}B^T WB Z = Z.
		$$
		Therefore,
		%\begin{subequations}
			\begin{align*}
			A_k^{-1} A VA &= A_k^{-1}AZ(Z^TAZ)^{-1}Z^T A \\
			&= Z(Z^TAZ)^{-1}Z^T A \\
			&= VA  \\
			&= VAA_k^{-1} A.
			\end{align*}
		%\end{subequations}
	\end{proof}
	
	\begin{theorem}
		Let $\sK$ be nonsingular with $A$ having nullity $k$, and let $W_k \in \mathbb{R}^{m \times m}$ be a rank-$k$ matrix such that $A + B^T W_k B$ is positive definite. The preconditioned operator $\sM_k^{-1}\sK$ has four distinct eigenvalues:
		\begin{itemize}
			\item $\lambda = -1$ with multiplicity $k$;
			\item $\lambda = 1$ with multiplicity $n-m+k$;
			\item $\lambda = \frac{1 \pm \sqrt{5}}{2}$, each with multiplicity $m-k$.
		\end{itemize}
	\end{theorem}

	\begin{proof}
		We consider the eigenvalue equations for the preconditioned system:
		\begin{subequations}
			\begin{align}
				\label{eqn:eigeq1}
				Ax + B^T y &= \lambda A_k x; \\
				\label{eqn:eigeq2}
				Bx &= \lambda S_k y.
			\end{align}
		\end{subequations}
		From \eqref{eqn:eigeq2} we obtain $y = \frac{1}{\lambda}S_k^{-1}Bx$. Substituting this into \eqref{eqn:eigeq1} and re-arranging yields
		\begin{equation}
			\label{eqn:nrg_first}
			A_k^{-1}Ax + \frac{1}{\lambda}A_k^{-1}B^T S_k^{-1}B x - \lambda x = 0.
		\end{equation}
		By Lemma \ref{lem:sc_add}, we can write
		\begin{align}
		\label{eqn:absb}
		\begin{aligned}
		A_k^{-1} B^T S_k^{-1} B &= A_k^{-1}B^T W_k B 
		\\ &\ \ \ \ \ \ + A_k^{-1}B^T  (BB^T)^{-1}B(A-AVA)B^T(BB^T)^{-1}B. 		
		\end{aligned}
		\end{align}
		As was discussed in the proof of Lemma \ref{lem:commute}, $VA$ is a projector onto $\ker(B)$, meaning that $I-VA$ is a projector onto $\range(B)$. Because $B^T(BB^T)^{-1}B$ is an orthogonal projector onto this subspace, we have
		$$
		(I-VA)B^T(BB^T)^{-1}B = I-VA.
		$$
		Similarly, $B^T(BB^T)^{-1}B(I-AV) = I-AV$. Thus, we can further simplify \eqref{eqn:absb}, using relations we developed in Lemma~\ref{lem:commute}:
		\begin{align*}
			A_k^{-1} B^T S_k^{-1} B &= A_k^{-1}B^T W_k B + A_k^{-1}(A-AVA)\\
			&= I - A_k^{-1}AVA \\
			&= I -VA.
		\end{align*}
		We can thus rewrite \eqref{eqn:nrg_first} as
		\begin{equation}
			\label{eqn:nrg_second}
			A_k^{-1}Ax - \frac{1}{\lambda}VA x + \left( \frac{1}{\lambda} -\lambda \right)x = 0.
		\end{equation}
		By Lemma \ref{lem:commute}, $A_k^{-1}A$ and $VA$ are commuting projectors; thus, they have the same eigenvectors. Because $VA$ has rank $n-m$ and $A_k^{-1}A$ has rank $n-k$, we have
		$$
		\range(VA) \subseteq \range(A_k^{-1}A) \textrm{ and } \ker(A_k^{-1}A) \subseteq \ker(VA).
		$$
		We now consider $x$ in the ranges/kernels of these projectors. 
		
		\textbf{Case I:} When $x \in \ker(A)$, \eqref{eqn:nrg_second} becomes
		\begin{equation}
		\label{eqn:xKerA}
		\left(\frac{1}{\lambda} - \lambda \right) x = 0.
		\end{equation}
		We note that $x$ cannot be zero, as \eqref{eqn:eigeq1} would necessarily imply $y = 0$. Thus, \eqref{eqn:xKerA} gives $k$ eigenvectors corresponding to each of the eigenvalues $\lambda = \pm 1$.
		
		\textbf{Case II:} When $x \in \range(VA)$ (and therefore also in $\range(A_k^{-1}A)$), \eqref{eqn:nrg_second} becomes
		$$
		\left(1 - \lambda \right) x = 0,
		$$
		which gives $n-m$ additional eigenvectors corresponding to the eigenvalue $\lambda = 1$.
		
		\textbf{Case III:} if $x \in \ker(VA)$ and $\range(A_k^{-1}A)$ (we know there are $m-k$ such vectors because the projectors commute), \eqref{eqn:nrg_second} becomes
		$$
		\left( 1 + \frac{1}{\lambda} - \lambda \right) x = 0,
		$$
		which gives the eigenvalues $\lambda = \frac{1 \pm \sqrt{5}}{2}$, each with geometric multiplicity $m-k$.
		
		Cases I-III account for all $n+m$ eigenvectors of $\sM_k^{-1}\sK$.
	\end{proof}

	\subsection{Schur complement approximations}
	\label{sec:prec_sc_approx}
	In practice, the blocks $A_k$ and $S_k$ of the ideal preconditioner $\sM_k$ defined in \eqref{eqn:Mk} are too expensive to invert exactly. While developing suitable approximation strategies for these terms often requires some knowledge of the problem at hand, we provide here two strategies for approximately inverting the Schur complement $S_k$.
	
	First, recall from Lemma \ref{lem:mrd_schur} that when $A$ has maximal nullity we have $S_k^{-1} = W_k$. Thus, when $A$ has high but not maximal nullity, it is reasonable to use an approximation of the form
	\begin{equation}
		\label{eq:sk_wki}
		S_k^{-1} \approx W_k + \beta I,
	\end{equation}
	where $\beta$ is a small positive value. We add the $\beta I$ term because if $A$ is not maximally rank-deficient then $W_k$ will be singular. We refer to this strategy as the ``WkI Schur complement approximation.''
	
	For our second strategy, recall that Lemma \ref{lem:sc_add} tells us that
	\begin{align*}
	S_k^{-1} &= W_k + (BB^T)^{-1}B(A-AVA)B^T(BB^T)^{-1} \\
	&= W_k + (BB^T)^{-1}BA\underbrace{(I-VA)}_{=:P}B^T(BB^T)^{-1}.
	\end{align*}
	Since $VA$ is a projector whose range is $\ker(B)$ and whose kernel is $\ker(Z^T A)$, the matrix $P=(I-VA)$ has range given by $\ker(Z^T A)$ and kernel given by $\ker(B)$. Thus, we consider replacing the projector $(I-VA)$ by the orthogonal projector onto $\range(B)$, defined by $P_B = B^T(BB^T)^{-1}B$. This matrix has the same kernel as $P$ but a different range, and has the advantage of yielding a considerably simpler second term, as we can write:
	\begin{align*}
		(BB^T)^{-1}B A P_B B^T (BB^T)^{-1} &= (BB^T)^{-1}B A B^T(BB^T)^{-1}B B^T (BB^T)^{-1} \\
		&= (BB^T)^{-1}B A B^T (BB^T)^{-1}.
	\end{align*}
	Thus, we can also consider the Schur complement approximation:
	\begin{equation}
		\label{eq:sk_bfbt}
		S_k^{-1} \approx W_k + (BB^T)^{-1}B A B^T (BB^T)^{-1}.
	\end{equation}
	We note that this modified second term is similar to the BFBT preconditioner proposed by Elman \cite{e99} for the Navier-Stokes equations; thus, we refer to this as the ``BFBT Schur complement approximation.''
	
	\section{Numerical experiments}
	\label{sec:numex}
	In this section we consider implementations of the block diagonal preconditioner described in Section \ref{sec:prec_ours}. All experiments are run in MATLAB R2021a on a commodity desktop PC. We report computation times for all experiments. The code is not optimized for efficiency and the measurements do not represent what would be possible with an optimized, state-of-the-art code base; they are included as a way to compare the computational costs of different approaches and validate our analytical observations.
	
	\subsection{Selection of weight matrix}
	Here we detail our general approach for choosing $W_k$. For simplicity, all our matrices $W_k$ are diagonal matrices with either 1 or 0 on the diagonal; thus, the augmented matrix $A_k$ is equal to $A$ in addition to $k$ terms of the form $b^Tb$, where $b$ is a single row of $B$. Hence, our task of selecting $W_k$ becomes the task of selecting which rows of $B$ to use in to augment $A$.
	
	We begin by forming a matrix $A_{drop}$ formed by eliminating very small elements of $A$ (for our purposes, we eliminate those matrix entries whose absolute values are less than machine epsilon times the largest magnitude entry in $A$). We then select rows of $B$ that increase the structural rank of $A_{drop}$ until the matrix $A_{drop}+ \sum_{i} b_i^T b_i$ has full structural rank. These selected rows of $b$ do not guarantee that the augmented matrix $A + \sum_{i} b_i^T b_i$ has full numerical rank or is sufficiently well-conditioned to avoid convergence problems, so in some cases we add additional rows of $B$; in this case, we greedily select the sparsest rows of $B$ to reduce fill-in of $A_k$.
	
	We note that, in general, this approach of selecting $W_k$ does \textit{not} guarantee a ``minimal-rank'' augmentation; that is, the rank of $W_k$ may be greater than the nullity of $A$. Finding a $W_k$ with rank exactly equal to the nullity of $A$ such that the augmented matrix $A_k$ is sufficiently well-conditioned to avoid numerical difficulty requires knowledge of the null-space of $A$ and of which vectors in $B$ will span that null space. That said, in many practical applications, for example in problems arising from discretizations of PDEs, some information on the discrete differential operators and their null space is often available and comes handy.
	
	\subsection{Constrained optimization problems}
	\subsubsection*{Problem statement}
	Given a positive semidefinite Hessian matrix $H \in \mathbb{R}^{n \times n}$, vectors $c \in \mathbb{R}^n$ and $b \in \mathbb{R}^m$ and a Jacobian matrix $J \in \mathbb{R}^{m \times n}$, consider the primal-dual pair of quadratic programs (QP) in standard form:
	\begin{subequations}
		\label{eq:QP}
		\begin{align}
			\min_x c^T x + \frac{1}{2}x^T H x \ \ \ \ &\textrm{s.t.} \ Jx = b, \ \ x \ge 0; \\
			\min_{x,y,z} b^T y - \frac{1}{2}x^T H x  \ \ \ \ &\textrm{s.t.} \ J^T y + z - Hx = c, z \ge 0,
		\end{align}
	\end{subequations}
	where $y$ and $z$ are vectors of Lagrange multipliers. In linear programming problems, we have $H=0$.
	
	Each step of a primal-dual interior-point method (IPM) to solve \eqref{eq:QP} requires solving a linear system of the form \cite{nw06}:
	\begin{equation*}
		\begin{bmatrix}
			H + X^{-1}Z & J^T \\
			J & 0
		\end{bmatrix}
		\begin{bmatrix}
			\Delta x \\
			\Delta y
		\end{bmatrix} = \begin{bmatrix}
			-c - Hx J^T y + \tau X^{-1}e \\
			b - Jx
		\end{bmatrix}.
	\end{equation*}
	See \cite{nw06} for full details. Some entries of the diagonal matrices $X$ and $Z$ approach zero as the IPM iterations proceed, so the leading block of the saddle-point matrix becomes increasingly ill-conditioned, with the largest magnitude entries occurring along the diagonal. Thus the leading block may become nearly singular or numerically singular, particularly if $H$ is singular.
	
	\subsubsection*{Description of test problems}
	We use an implementation of the predictor-corrector algorithm of Mehrotra \cite{m92}. The matrices for linear programming problems were obtained from the Sparse Suite matrix collection \cite{dh11}, and the quadratic programming problems are from TOMLAB\footnote{Test matrices available at https://tomopt.com/tomlab/.}. A summary of the test suite of LP problems used in our experiments is given in Table \ref{table:lpProblems}.
	
	\begin{table}[tbh!]
		\centering
		\begin{tabular}{ c c c c }
			\hline
			Problem ID & $m$ & $n$ & $\nnz(\sK)$ \\
			\hline
			\texttt{lp\_80bau3b} & 2,262 & 12,061 & 35,325 \\
			\texttt{lp\_bandm} & 305 & 472 & 2,966 \\
			\texttt{lp\_capri} & 271 & 482 & 2,378 \\
			\texttt{lp\_finnis} & 497 & 1,064 & 3,824 \\
			\texttt{lp\_fit1p} & 627 & 1,677 & 11,545 \\
			\texttt{lp\_ganges} & 1,309 & 1,706 & 8,643 \\
			\texttt{lp\_lofti} & 153 & 366 & 1,502 \\
			\texttt{lp\_maros\_r7} & 3,136 & 9,408 & 154,256 \\
			\texttt{lp\_osa\_14} & 2,337 & 5,497 & 371,894 \\
			\texttt{lp\_osa\_30} & 4,350 & 104,374 & 708,862 \\
			\texttt{lp\_pilot87} & 2,030 & 6,680 & 81,629 \\
			\texttt{lp\_scfxm1} & 330 & 600 & 3,332 \\
			\texttt{lp\_scsd8} & 397 & 2,750 & 11,334 \\
			\texttt{lp\_stair} & 356 & 614 & 4,617 \\
			\texttt{lp\_standmps} & 467 & 1,274 & 5,152 \\
			\texttt{lp\_stocfor2} & 2,157 & 3,045 & 12,402 \\
			\texttt{lp\_truss} & 1,000 & 8,806 & 36,642 \\
			\texttt{lp\_vtp\_base} & 198 & 346 & 1,397 \\
			\hline
		\end{tabular}
		\caption{Summary of linear programming (LP) problems used in numerical experiments. The value $nnz(\sK)$ gives the number of nonzeros arising in the saddle-point system at each interior point method (IPM) iteration.}
		\label{table:lpProblems}
	\end{table}

	\subsubsection*{Comparison of different augmentation and approximation strategies}
	In this experiment we consider preconditioners of the form
	\begin{equation}
		\label{eq:approxA}
		\sM = \begin{bmatrix}
			\tilde{A}_{aug} & 0 \\
			0 & B\hat{A}_{aug}^{-1}B^T,
		\end{bmatrix}
	\end{equation}
	where $\tilde{A}_{aug}$ and $\hat{A}_{aug}$ are approximations (potentially the same approximation) of an augmented leading block $A$. Our experiments are on matrices that arise while applying an interior-point method on an LPs, so the leading block $A$ is diagonal. We consider three augmentation strategies:
	\begin{enumerate}
		\item Partial augmentation: we take $A_{aug} = A + B^T W_k B$, where we form $W_k$ by selecting just enough rows of $B$ such that $A_{drop} + B^T W_k B$ has full structural rank, where $A_{drop}$ is the matrix obtained by setting to zero all elements of $A$ with absolute value less than or equal to machine-epsilon times the largest absolute magnitude value of $A$.
		\item Full augmentation: we take $A_{aug} = A + B^T B$.
		\item Identity augmentation: we take $A_{aug} = A + \rho I$, for some positive $\rho$.
	\end{enumerate}
	For $A_{aug}$ arising from partial and full augmentation, we consider three approximations for $\tilde{A}_{aug}$ and $\hat{A}_{aug}$ in \eqref{eq:approxA}:
	\begin{enumerate}
		\item Ideal approximation (ID): $\tilde{A}_{aug} = \hat{A}_{aug} = A_{aug}$. (This is too expensive to use in practice but we include it here for comparison purposes.)
		\item Diagonal approximation (D): $\tilde{A}_{aug} = \hat{A}_{aug} = \diag(A_{aug})$.
		\item Incomplete Cholesky approximation (IC) : $\tilde{A}_{aug} = \IC(A_{aug})$ and $\hat{A}_{aug} = \diag(A_{aug})$. We use ICT with drop tolerance of $0.01$.
	\end{enumerate}
	For the identity-based augmentation, the matrix $A_{aug}$ is diagonal, so we solve it exactly (that is, $\tilde{A}_{aug} = \hat{A}_{aug} = A_{aug}$).
	
	\begin{table}[tbh!]
		\centering
		\resizebox{\textwidth}{!}{%
			\begin{tabular}{|l|ccc|ccc|c|}
				\hline
				\multirow{2}{*}{Problem ID} & \multicolumn{3}{c|}{Partial}                            & \multicolumn{3}{c|}{Full}                            & Identity \\ \cline{2-8} 
				& \multicolumn{1}{c|}{ID} & \multicolumn{1}{c|}{D} & IC & \multicolumn{1}{c|}{ID} & \multicolumn{1}{c|}{D} & IC & ID \\ \hline
				\texttt{80bau3b}& \multicolumn{1}{c|}{5 (0.03)} & \multicolumn{1}{c|}{22 (0.03)} & 230 (0.02) & \multicolumn{1}{c|}{18 (2.0)} & \multicolumn{1}{c|}{122 (0.02)} & 254 (0.01) & 43 (0.02) \\ \hline
				\texttt{maros\_r7} & \multicolumn{1}{c|}{22 (3.7)} & \multicolumn{1}{c|}{22 (0.2)} & 56 (0.1) & \multicolumn{1}{c|}{2 (2.2)} & \multicolumn{1}{c|}{19 (0.1)} & 26 (0.1) & 11 (0.1) \\ \hline
		\end{tabular}}
		\caption{MINRES iteration counts for partial, full and identity-augmentation preconditioners for the \texttt{lp\_80bau3b} and \texttt{lp\_maros\_r7} problems, using various block approximation strategies (ID=ideal, D=diagonal, IC=incomplete Cholesky). Time per iteration (in seconds) is given in parentheses.}
		\label{table:allPreconsIter}
	\end{table}
	
	\begin{table}[tbh!]
		\centering
		\resizebox{\textwidth}{!}{%
			\begin{tabular}{|l|ccc|ccc|}
				\hline
				\multirow{2}{*}{Problem ID} & \multicolumn{3}{c|}{Partial augmentation}                            & \multicolumn{3}{c|}{Full augmentation}                            \\ \cline{2-7} 
				& \multicolumn{1}{c|}{Rank($W$)} & \multicolumn{1}{c|}{$\nnz(A_W)$} & $\nnz(\IC(A_W))$ & \multicolumn{1}{c|}{Rank($W$)} & \multicolumn{1}{c|}{$\nnz(A_W)$} & $\nnz(\IC(A_W))$ \\ \hline
				\texttt{80bau3b} & \multicolumn{1}{c|}{2} & \multicolumn{1}{c|}{12,249} & 12,101 & \multicolumn{1}{c|}{2,262} &  \multicolumn{1}{c|}{456,943} & 14,183 \\ \hline
				\texttt{maros\_r7} & \multicolumn{1}{c|}{2,511} & \multicolumn{1}{c|}{1,101,752} & 31,343 & \multicolumn{1}{c|}{3,136} & \multicolumn{1}{c|}{1,230,928} & 10,761 \\ \hline
		\end{tabular}}
		\caption{Comparison of memory usage for partial and full augmentation for the \texttt{lp\_80bau3b} and \texttt{lp\_maros\_r7} problem.}
		\label{table:allPreconsMem}
	\end{table}
	
	We use matrices that arise from IPMs on the test problems \texttt{lp\_80bau3b} and \texttt{lp\_maros\_r7}. Iteration counts and time per iteration are given in Tables \ref{table:allPreconsIter} and \ref{table:allPreconsMem}.
	
	We observe that for \texttt{lp\_80bau\_3b}, the partial augmentation preconditioner outperforms the full augmentation preconditioner in terms of both iteration count and memory usage. This is because the leading block of this matrix is only mildly rank-deficient, so we only need a low-rank augmentation to make it nonsingular (which leads to a much sparser augmented matrix than the full augmentation); additionally, when we fully augment this matrix we are far away from the ``ideal'' amount of augmentation (i.e., the rank of augmentation that would yield a constant number of eigenvalues in an ideally-preconditioned iterative solver) because the leading block is nowhere near lowest-rank.
	
	In contrast, the leading block for \texttt{lp\_maros\_r7} is highly rank-deficient, as even the minimal amount of augmentation to obtain a structurally nonsingular leading block requires using most of the rows of $B$ (2,511, when $m$ for this problem is 3,136). And we observe that, in cases like these where the nullity of the leading block is high, we are close enough to the lowest-rank case that full augmentation performs well. In this case, it actually performs better than the partial augmentation in terms of iteration counts and computation time because the fully augmented leading block is more well-conditioned than the partially augmented leading block. Recall that our procedure for choosing $W_k$ only looks at structural rank, and does not guarantee that the augmented matrix is actually nonsingular (so we may still encounter numerical difficulties without further augmentation).
	
	Finally, we note that the incomplete Cholesky approximation strategy is less effective than the diagonal approximation strategy. One reason for this is that by the time IPM matrices are singular, the largest magnitude entries tend to occur along the diagonal; thus, a diagonal leading block approximation is generally effective (as we will see in the next set of experiments). The other is that when we used the incomplete Cholesky in the leading block, we avoided using the inverse of the incomplete Cholesky factors in the Schur complement approximation to avoid introducing too much computational expense. Thus, the Schur complement approximation is not equal to $B \tilde{A}_{aug}^{-1}B^T$ (where $\tilde{A}_{aug}$ is the selected leading block approximation); and as we saw in Section \ref{sec:prec_ours}, this has an impact on the theoretical properties of the preconditioned operator.
	
	\subsubsection*{Running partial augmentation preconditioners on LP test suites}
	Here we consider preconditioning the complete set of problems described in Table \ref{tab:allLP}. The matrices reported below are the first matrices for which the IPM generates a matrix with a numerically singular leading block. We consider the partial augmentation preconditioner of the form \eqref{eq:approxA} with the diagonal leading block approximation strategy: that is, we define $P_D$ using $\tilde{A}_{aug} = \hat{A}_{aug} = \diag(A_{aug})$. In all cases, we select $W_k$ by augmenting $A$ until the matrix $A_{drop} + B^T W_kB$ is structurally nonsingular. MINRES solver tolerance is set to a relative residual norm of $10^{-8}$.
	
	\begin{table}[tbh!]
		\centering
			\begin{tabular}{|c|c|c|cc|}
				\hline
				\multirow{2}{*}{Problem ID} & \multirow{2}{*}{$\rank(W_k)$} & \multirow{2}{*}{$\nnz(A_k)$} & \multicolumn{2}{c|}{$P_D$}\\ \cline{4-5} 
				&  &  & \multicolumn{1}{c|}{Iters} & Time per iter \\ \hline
				\texttt{80bau3b} & 1 & 12,117 & \multicolumn{1}{c|}{20} & 0.02 \\ \hline
				\texttt{bandm} & 5 & 1,444  & \multicolumn{1}{c|}{40} & 0.003 \\ \hline
				\texttt{capri} & 13 & 2,230 & \multicolumn{1}{c|}{67} & 0.003 \\ \hline
				\texttt{finnis} & 29 & 11,184 & \multicolumn{1}{c|}{77} & 0.006 \\ \hline
				\texttt{fit1p} & 5 & 2,545 & \multicolumn{1}{c|}{28} & 0.06 \\ \hline
				\texttt{ganges} & 88 & 2,690 & \multicolumn{1}{c|}{41} & 0.01 \\ \hline
				\texttt{lofti} & 13 & 966 & \multicolumn{1}{c|}{194} & 0.001 \\ \hline
				\texttt{maros\_r7} & 64 & 73,102 & \multicolumn{1}{c|}{26} & 0.2 \\ \hline
				\texttt{osa\_14} & 34 & 98,459,317 & \multicolumn{1}{c|}{171} & 0.06 \\ \hline
				\texttt{osa\_30} & 4 & 354,880,632 & \multicolumn{1}{c|}{80} & 0.1 \\ \hline
				\texttt{pilot87} & 5 & 133,798 & \multicolumn{1}{c|}{37} & 0.2 \\ \hline
				\texttt{scfxm1} & 1 & 840 & \multicolumn{1}{c|}{32} & 0.003 \\ \hline
				\texttt{scsd8} & 36 & 16,826 & \multicolumn{1}{c|}{6} & 0.003 \\ \hline
				\texttt{stair} & 33 & 9,994 & \multicolumn{1}{c|}{11} & 0.006 \\ \hline
				\texttt{standmps} & 2 & 557,906 & \multicolumn{1}{c|}{65} & 0.004 \\ \hline
				\texttt{stocfor2} & 61 & 3,411 & \multicolumn{1}{c|}{9} & 0.1 \\ \hline
				\texttt{truss} & 15 & 18,468 & \multicolumn{1}{c|}{34} & 0.005 \\ \hline
				\texttt{vtp\_base} & 10 & 3,126 & \multicolumn{1}{c|}{125} & 0.002 \\ \hline
		\end{tabular}
		\caption{MINRES iteration counts and time per iteration (in seconds) of the partial augmentation preconditioners with diagonal approximations of $A_k$.}
		\label{tab:allLP}
	\end{table}

	Eigenvalues of the preconditioned operator $P_D^{-1}\sK$ are shown in Figure \ref{fig:fit1p_eig} for \texttt{lp\_fit1p} problem. There is strong clustering of eigenvalues near $1, \frac{1 \pm \sqrt{5}}{2}$.

	\begin{figure}[tbh!]
		\centering
		\includegraphics[width=.8\linewidth]{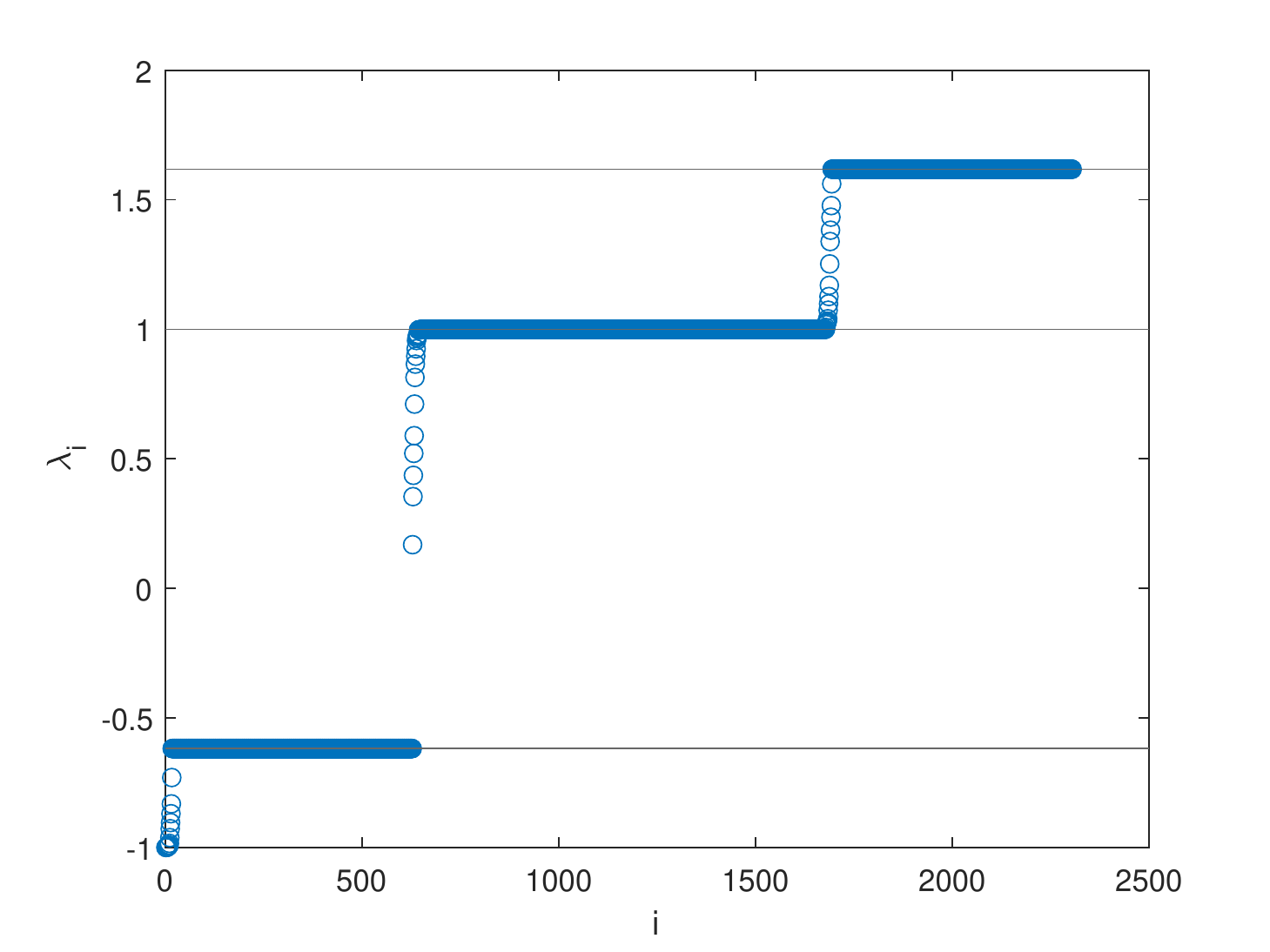}  
		\caption{Eigenvalues of preconditioned operator $P_D^{-1}\sK$ for matrix arising in the IPM solution of the \texttt{lp\_fit1p} problem. Horizontal lines are shown at $y=\pm1, \frac{1 \pm \sqrt{5}}{2}$.}
		\label{fig:fit1p_eig}
	\end{figure}
	
	\subsubsection*{Using preconditioned MINRES iterations in an IPM}
	Here we consider using preconditioned inner solves in an IPM solver. For our test problems, we use the LP \texttt{lp\_stocfor2} and the TOMLAB QP problem 37 (which has $m=490$; $n=1275$; 3,288 nonzeros in the Jacobian matrix; and 290 in the Hessian). Our preconditioning approach at each iteration is as follows:
	\begin{itemize}
		\item If the leading block $A$ is nonsingular, we use the preconditioner
		$$
		\sM_{LP} = \begin{bmatrix}
			A & 0 \\
			0 & BA^{-1}B^T
		\end{bmatrix}
		$$
		for the LP (recall that in this context $A$ is diagonal), and
		$$
		\sM_{QP} = \begin{bmatrix}
			\IC(A) & 0 \\
			0 & B(\diag(A))^{-1}B^T
		\end{bmatrix}
		$$
		for the QP, with an ICT drop tolerance of $0.01$.
		\item If the leading block $A$ is singular, we select the lowest-rank $W_k$ to make $A_{drop} + B^T W_kB$ nonsingular, and use the preconditioner
		$$
		\sM = \begin{bmatrix}
			\diag(A_k) & 0 \\
			0 & B(\diag(A_k))^{-1}B^T
		\end{bmatrix}.
		$$
	\end{itemize}
	We solve the IPM to a duality gap tolerance of $10^{-6}$ and use an inner tolerance of $10^{-7}$ for the MINRES solves.
	
	We see that for both problems, using inexact solves results in modestly more IPM iterations, as we would expect. For the LP, the leading block was nonsingular for the first 21 iterations and numerically singular for the final 10. For the QP, the leading block was nonsingular for the first 22 iterations and singular for the last 16. Notice that the average MINRES iteration counts are correspondingly higher for the QP. This is because, at the LP steps with a nonsingular leading block, we were able to use an ideal preconditioner because the leading block is diagonal, and convergence was always achieved in roughly three iterations. Additionally, the nonzero Hessian in the QP has some additional terms in the leading block that are dropped in the diagonal leading block approximation once the leading block becomes singular.
	
	\begin{table}[tbh!]
		\resizebox{\textwidth}{!}{%
			\begin{tabular}{|ll|c|ccc|}
				\hline
				\multicolumn{2}{|c|}{Problem} & Direct inner solve & \multicolumn{3}{c|}{MINRES inner solve} \\ \hline
				\multicolumn{1}{|c|}{\multirow{2}{*}{ID}} & \multirow{2}{*}{Type} & \multirow{2}{*}{IPM iterations} & \multicolumn{1}{c|}{\multirow{2}{*}{IPM iterations}} & \multicolumn{2}{c|}{Inner iters (average)}    \\ \cline{5-6} 
				\multicolumn{1}{|c|}{} &  & & \multicolumn{1}{c|}{} & \multicolumn{1}{c|}{Predictor} & Corrector \\ \hline
				\multicolumn{1}{|c|}{\texttt{stocfor2}} & LP & 27 & \multicolumn{1}{c|}{31}& \multicolumn{1}{c|}{4.1} & 4.1 \\ \hline
				\multicolumn{1}{|c|}{\texttt{TOMLAB37}}  &  QP & 31 & \multicolumn{1}{c|}{38} & \multicolumn{1}{c|}{35.1} & 36.6 \\ \hline
		\end{tabular}}
		\caption{Comparison of IPM iterations using a direct vs. preconditioned MINRES solver for the inner linear system solves. Average number of inner MINRES iterations are reported for both the predictor and corrector steps.}
	\end{table}
		
	\subsubsection*{Testing different block approximation strategies}
	
	Here we test the WkI Schur complement approximation strategy (see Eq. \eqref{eq:sk_wki}). We use a matrix that arises at the 20th iteration of the IPM solution for the LP \texttt{maros\_r7} and use $\beta = 0.5$. As we have seen in our earlier LP experiments, by the time the IPM iterations have advanced enough to create a numerically singular leading block, the diagonal has enough large entries that the augmented matrix $A_k$ is mostly diagonally dominant. Thus, using $\diag(A_k)$ is often effective in approximating $A_k$. We include comparisons between the preconditioners in which:
	\begin{itemize}
		\item $A_k$ approximated by $\diag(A_k)$ and $S_k^{-1}$ is approximated by \\ $B\diag(A_k)^{-1}B^T$ (the preconditioner $P_D$ explored in the previous set of experiments);
		\item $A_k$ is approximated by $\diag(A_k)$ and $S_k^{-1}$ is approximated by $W_k + \beta I$ (``Diagonal+WkI'' or ``D+WkI'').
	\end{itemize}
	 For this experiment, our weight matrix $W_k$ has rank 2,911 (the minimum required to achieve structural nonsingularity of $A_{drop} + B^T W_kB$).
	
	A convergence plot is shown in Figure \ref{fig:wki_sc}. The $P_D$ preconditioner converges in 11 iterations and 1.4 seconds (0.1 seconds per iteration), and the Diagonal+WkI preconditioner in 102 iterations and 0.18 seconds (0.0018 seconds per iteration). While this is a significantly higher iteration count, we notice that this preconditioner is extremely cheap (in that it is fully diagonal) and thus results in faster computational time overall. We note that a basic Jacobi iteration on the original system (or Jacobi on the leading block combined with the WkI approximation of the Schur complement) does not lead to convergence. Thus, the leading block augmentation has utility in arriving at this surprisingly simple-looking preconditioner.
	
	\begin{figure}[tbh!]
		\centering
		\includegraphics[width=.8\linewidth]{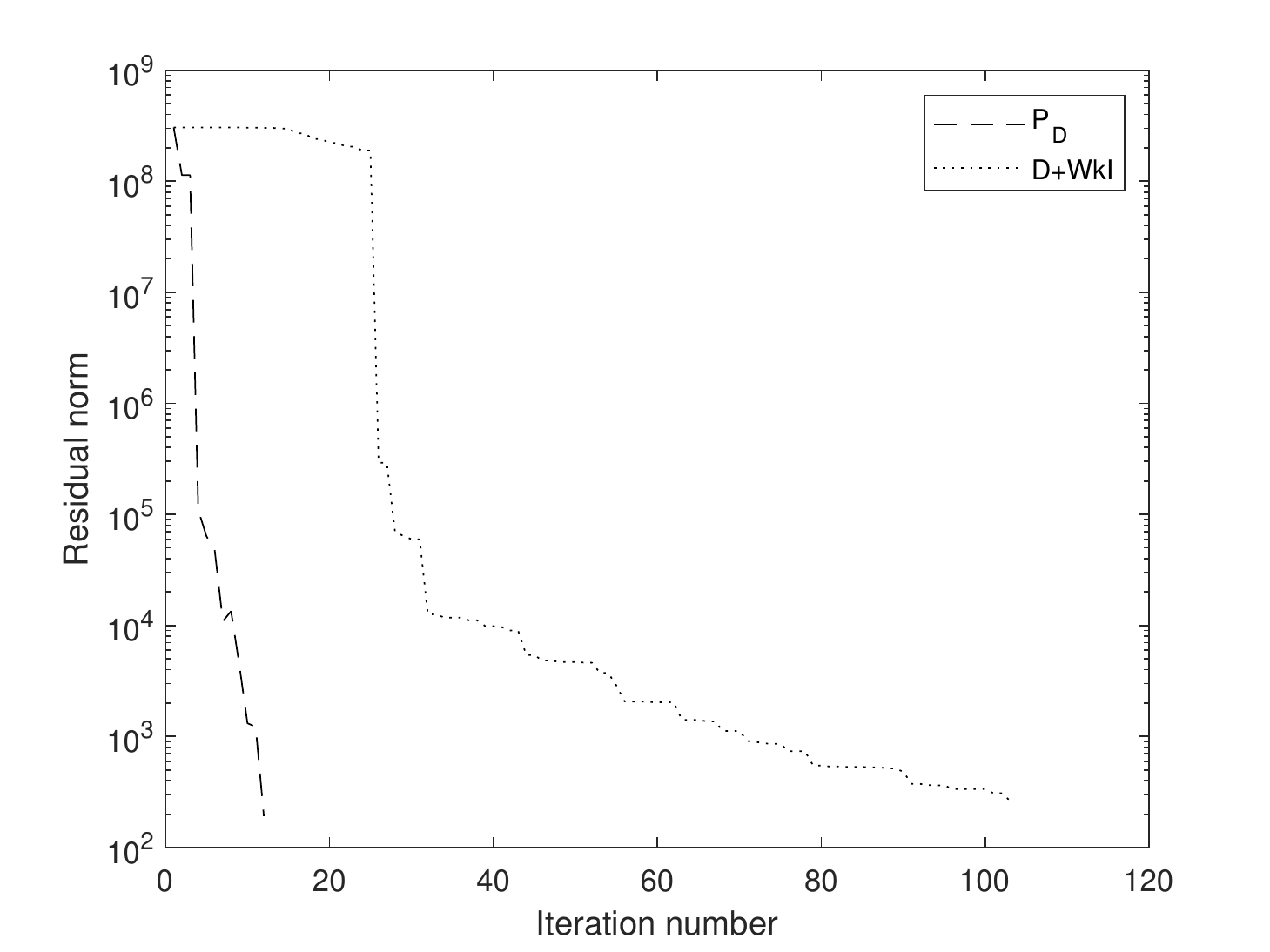}  
		\caption{Comparison of block approximation strategies (diagonal leading block + $B(\diag(A_k))^{1}B^T$ Schur complement; Diagonal leading block+WkI Schur complement) for a matrix arising from an IPM on the \texttt{lp\_maros\_r7} problem.}
		\label{fig:wki_sc}
	\end{figure}
		
	\subsection{A geophysical inverse problem}
	
	\subsubsection*{Problem statement}
	Here we consider the example of a geophysical inverse problem described in~\cite{h00}, which involves recovering a model based on observations of a field. The regularized problem is defined by
	\begin{align*}
		\min_{m,u} \ \ & \frac{1}{2}||Qu - b||^2 + \frac{\beta}{2}||W(m - m_{ref})||^2 \\
		\textrm{s.t.} \ \ &A(m)u = q,
	\end{align*}
	where $\beta$ is a regularization parameter, $m$ is a model, $m_{ref}$ is a reference model, $W$ is a weighting matrix, and $A$ is a large, sparse, nonsingular matrix that encodes the model conditions of the field being considered. If Gauss-Newton iterations are used, the linear system to be solved at each step takes the form
	\begin{equation*}
		\begin{bmatrix}
			Q^TQ & 0 & A^T \\
			0 & \beta W^T W & G^T \\
			A & G & 0
		\end{bmatrix}
		\begin{bmatrix}
			\delta u \\
			\delta m \\
			\delta \lambda
		\end{bmatrix} = 
		-\begin{bmatrix}
			r_u \\
			r_m \\
			r_{\lambda}
		\end{bmatrix},
	\end{equation*}
	with $G$ being the Jacobian of $A$. In the typical case of sparse observations, $G$ is sparse and $Q^TQ$ has high nullity.

	\subsubsection*{Testing different block approximation strategies}
	In this experiment we test the BFBT Schur complement approximation strategy (Eq. \eqref{eq:sk_bfbt}). We set the regularization parameter $\beta = 10^{-3}$. The leading block is highly singular, so we augment $A$ by all of $B$ to avoid numerical difficulties (as simply augmenting by enough rows of $B$ to make the augmented matrix structurally nonsingular still leads to a matrix that is highly ill-conditioned).
	
	Recall that the BFBT Schur complement approximation requires two solves for $B B^T$. Fortunately, for the geophysics problem, this term is sparse and banded. Thus, in computing this approximation, we will solve exactly for the $B B^T$ terms.
	
	We note that the augmented matrix $A+B^TB$ has an interesting structure, as we can see in Figure \ref{fig:geophys_vis_Ak}: if we partition the matrix four blocks with the (1,1)-block of size $m$ and the (2,2)-block of size $n-m$, we observed that the (1,1)- and (2,2)-blocks are banded (e.g., for a problem with $m=9,261$ and $n=17,261$, the bandwidths are 848 and 421, respectively), and can therefore be solved less expensively than the entire matrix $A + B^T B$. Thus, we can use block Jacobi to approximately solve $A_k$. Because stationary methods are often not especially effective as preconditioners, we will instead use block Jacobi as a preconditioner for an inner preconditioned conjugate gradient (PCG) solver for $A_k$.
	
	\begin{figure}[tbh!]
		\centering
		\includegraphics[width=.8\linewidth]{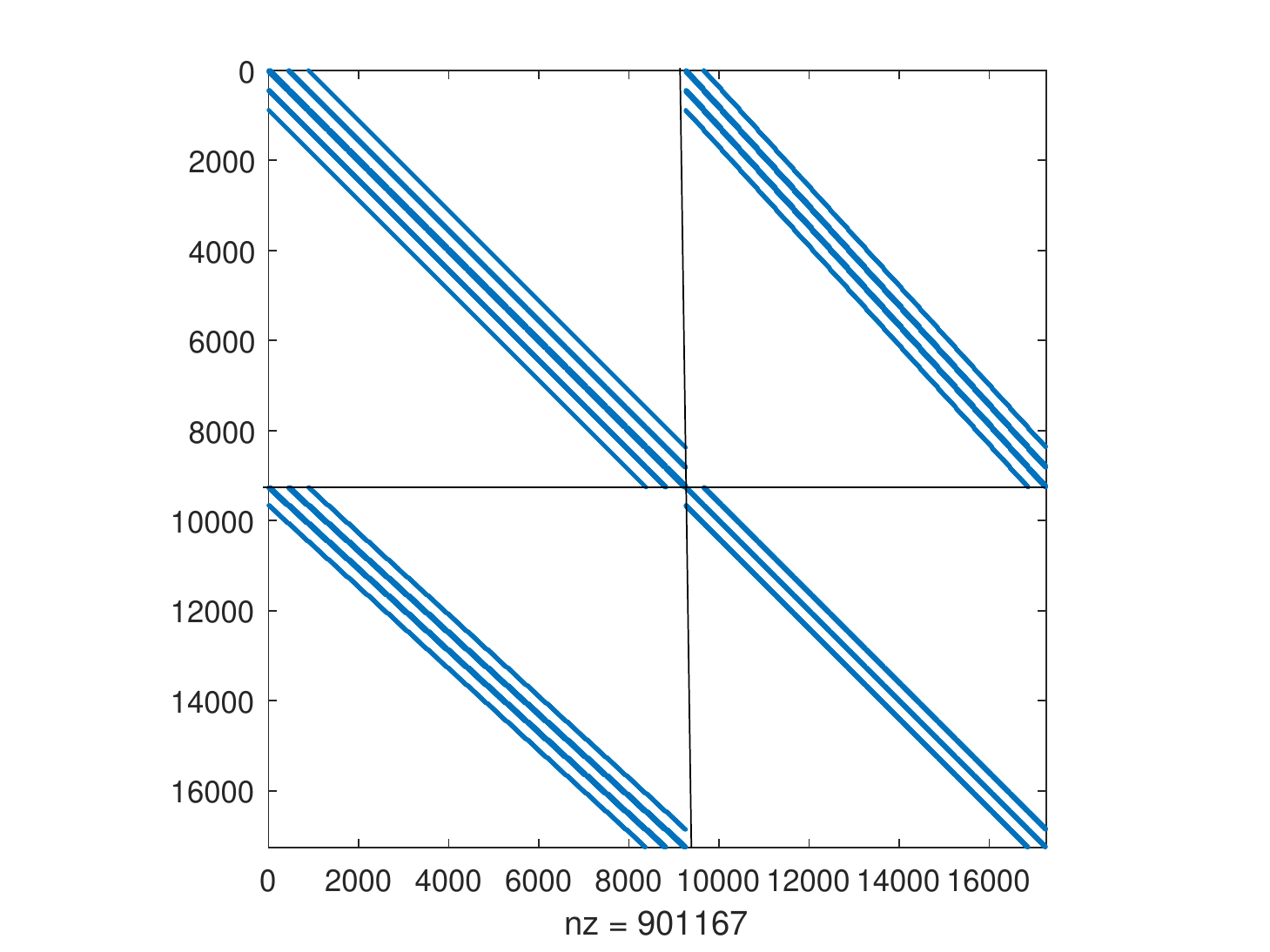}  
		\caption{Sparsity pattern of $A_k=A+B^T B$ for a geophysics problem with $m=9,261$ and $n=17,261$.}
		\label{fig:geophys_vis_Ak}
	\end{figure}

	Thus, in these experiments, we compare the preconditioners in which:
	\begin{itemize}
		\item $A_k$ is inverted exactly (which is generally not practical for large problems but is included here for validation and comparison), and $S_k^{-1}$ is approximated with the BFBT approximation. We denote this by ``Akinv+BFBT.''
		\item $A_k$ is inverted approximately using CG to an inner tolerance of $0.1$, with block Jacobi as a preconditioner, and $S_k^{-1}$ is approximated by the BFBT approximation. We denote this by ``CG+BFBT.''
	\end{itemize}
	We use MINRES for the Akinv+BFBT preconditioner and FGMRES(30) for the CG+BFBT.
	
	\begin{table}[tbh!]
		\centering
		\begin{tabular}{|c|c|cc|cc|}
			\hline
			\multirow{2}{*}{$m$} & \multirow{2}{*}{$n$} & \multicolumn{2}{c|}{Akinv+BFBT} &  \multicolumn{2}{c|}{CG+BFBT} \\ \cline{3-6} 
			& & \multicolumn{1}{c|}{Iters} & Time per iter & \multicolumn{1}{c|}{Iters} & Time per iter  \\ \hline
			2,197 & 3,195 & \multicolumn{1}{c|}{6} & 0.21 & \multicolumn{1}{c|}{9} & 0.20  \\ \hline
			4,913 & 9,009 & \multicolumn{1}{c|}{6} & 1.07 & \multicolumn{1}{c|}{10} & 0.76 \\ \hline
			\multicolumn{1}{|c|}{9,261} & \multicolumn{1}{c|}{17,261} & \multicolumn{1}{c|}{8} & \multicolumn{1}{c|}{2.87} & \multicolumn{1}{c|}{10} & 2.26  \\ \hline
		\end{tabular}
		\caption{Results (solver iteration counts and time per iteration) geophysics problems of varying size. Akinv+BFBT = exact solve for $A_k$, BFBT approximation for $S_k$; CG+BFBT = block Jacobi preconditioned CG for $A_k$, BFBT for $S_k$.}
		\label{tab:geophys}
	\end{table}

	Results are shown in Table \ref{tab:geophys}. The Akinv+BFBT preconditioner performs well in terms of iteration count, but includes a very expensive term in the $A_k$ solve. We note, however, that the number of preconditioned iterations is very close to what we would expect of the ideal preconditioner (with exact solves for both $A_k$ and $S_k$), which highlights the effectiveness of the BFBT Schur complement approximation for this problem. The CG+BFBT preconditioner achieves similar convergence to the Akinv+BFBT -- in particular, the number of iterations appears to be independent of problem size -- and is modestly less expensive per iteration in terms of compute time (we avoid the direct solve for $A_k$, but have some added expense from the inner CG solves and additional orthogonalization for FGMRES). On average, the inner PCG solves required 28.7 iterations for the first  test problem (with $m=2,197$ and $n=3,195$), 35.1 iterations for the second problem (with $m=4,913$ and $n=9,009$), and 35.8 iterations for the third (with $m=9,261$ and $n=17,261$). For larger problems, we speculate that CG+BFBT will outperform Akinv+BFBT by larger margins.

	\section{Concluding remarks}
	\label{sec:conclusions}
	We have developed a block diagonal preconditioner for saddle-point systems with a singular leading block. We showed how, by augmenting $A$ with a weight matrix of just high enough rank to overcome its nullity, we yield a preconditioned operator with a small, fixed number of distinct eigenvalues. In doing so, we have closed a gap in the existing literature, in analyzing a preconditioning approach for a scenario where the leading block of the saddle-point matrix is neither full rank nor does it have nullity equal to the number of rows of $B$.

%	Specifically, we have considered preconditioners of the form
%	$$
%	\tilde{\sM} = \begin{bmatrix}
%		\tilde{A}_k & 0 \\
%		0 & \tilde{S}_k
%	\end{bmatrix},
%	$$
%	where $\tilde{A}_k$ and $\tilde{S}_k$ are some approximations of $A_k$ and $S_k$, respectively. While we have in places included the ideal case $\tilde{A}_k = A_k$ and $\tilde{S}_k = S_k$ for comparison purposes, these are too expensive to be used in a practical solver.
	
	Specifically, we have considered block preconditioners that are based on approximating the augmented leading block of the saddle-point matrix and the augmented Schur complement.
	Typically, the construction of the weight matrix $W_k$ and the selection of effective approximations may be guided by the problem at hand (for example, in cases where the matrix blocks and Schur complement arise from well-studied discrete differential operators). We have provided some general approaches that may work for different problems. For $A_k$, we have included diagonal (for LPs), incomplete Cholesky (for QPs), block Jacobi and inner PCG iterations (for geophysics); and for $S_k$,  the $B(\diag(A_k))^{-1}B^T$ and WkI approximations (for the optimization problems), and the BFBT approximation (for the geophysics problem). 
	
	We have restricted ourselves to diagonal weight matrices with all ones and zeros along the diagonal and have described a method that looks only at the structural rank of a modified augmented matrix. Future work may include more sophisticated choices of the weight matrix, which may in turn yield faster convergence.

	% Bibliography
	\bibliographystyle{abbrv}
	\bibliography{bg22-ref}

\end{document}